\documentclass[12pt]{amsart}
\usepackage{amssymb}
\usepackage{amsmath,amsfonts,amssymb,upgreek}

\usepackage[breaklinks]{hyperref}
\usepackage{amscd}
\usepackage{mathtools}

\theoremstyle{thmrm}
\usepackage{tikz}
\usepackage{enumerate}
\usetikzlibrary{calc,decorations.markings}
\usepackage{tikz-cd}
\usepackage{graphicx}

\theoremstyle{plain}
\newtheorem{no}{Notation}[section]
\newtheorem{Prob}{Problem}
\newtheorem{thm}{Theorem}[section]
\newtheorem{lemma}[thm]{Lemma}
\newtheorem{prop}[thm]{Proposition}
\newtheorem*{Algo}{Algorithm}
\newtheorem{cor}[thm]{Corollary}

\newtheorem{question}[thm]{Question}
\newtheorem{defn}[thm]{Definition}
\newtheorem{example}[thm]{Example}
\newtheorem*{ack}{Acknowledgement}

\theoremstyle{definition}
\newtheorem{remark}[equation]{Remark}

\newcommand{\Z}{\operatorname{Z} }

\newcommand{\im}{\operatorname{Im} }

\newcommand{\Fix}{\operatorname{Fix}}

\newcommand{\Hol}{\operatorname{Hol}}
\newcommand{\Id}{\operatorname{Id}}

\newcommand{\I}{\operatorname{I} }

\newcommand{\Aut}{\operatorname{Aut} }

\newcommand{\Hom}{\operatorname{Hom} }

\newcommand{\Ker}{\operatorname{Ker} }

\newcommand{\Ann}{\operatorname{Ann} }

\newcommand{\IM}{\operatorname{Im} }

\numberwithin{equation}{section}

\begin{document}

\title{Relative Rota-Baxter groups and skew left braces}

\author{Nishant Rathee}
\address{Department of Mathematical Sciences, Indian Institute of Science Education and Research (IISER) Mohali, Sector 81, SAS Nagar, P O Manauli, Punjab 140306, India} 
\email{nishantrathee@iisermohali.ac.in}

\author{Mahender Singh}
\address{Department of Mathematical Sciences, Indian Institute of Science Education and Research (IISER) Mohali, Sector 81, SAS Nagar, P O Manauli, Punjab 140306, India} 
\email{mahender@iisermohali.ac.in}

\subjclass[2010]{17B38, 16T25}
\keywords{Isoclinism; relative Rota-Baxter group, Rota-Baxter group; skew left braces; Yang-Baxter equation}

\begin{abstract}
Relative Rota-Baxter groups are generalisations of Rota-Baxter groups and introduced recently in the context of Lie groups. In this paper, we explore connections of relative Rota-Baxter groups with skew left braces, which are well-known to give non-degenerate set-theoretic solutions of the Yang-Baxter equation. We prove that every  relative Rota-Baxter group gives rise to a skew left brace, and conversely, every skew left brace arises from a relative Rota-Baxter group. It turns out that there is an isomorphism between the two categories under some mild restrictions. We propose an efficient GAP algorithm, which would enable the computation of relative Rota-Baxter operators on finite groups. In the end, we introduce the notion of isoclinism of relative Rota-Baxter groups and prove that an isoclinism of these objects induces an isoclinism of corresponding skew left braces.
\end{abstract}	

\maketitle

\section{Introduction}
The quantum Yang-Baxter equation is a fundamental equation arising in mathematical physics that forms the basis of the theory of quantum groups. In \cite[Section 9]{D92}, Drinfeld proposed to investigate set-theoretical solutions of the quantum Yang-Baxter equation. A set-theoretical solution is defined as a set $X$ together with a map $R: X \times X \rightarrow X \times X$ that satisfies the equation 
$$R^{12} R^{13} R^{23}= R^{23} R^{13} R^{12},$$
 where each $R^{ij}:X \times X \times X \rightarrow X \times X \times X$ is a map that acts as $R$ on the $i$-th and the $j$-th component and as identity on the remaining component. Setting $R(x, y)= (f_x(y), g_y(x))$, if the maps $R$, $f_x$, and $g_x$ are bijections for all $x \in X$, then the solution is said to be non-degenerate. If $R^2=\Id _{X\times X}$, then the solution is referred to as involutive. The complete classification of set-theoretical solutions of the quantum Yang-Baxter equation is a wide open problem and has attracted considerable interest in the last decade.
\par

Rump \cite{WR07} introduced left braces as generalisations of Jacobson radical rings and showed that they give rise to involutive set-theoretical solutions of the Yang-Baxter equation. Guarnieri and Vendramin  \cite{GV17} later generalized the concept to skew left braces, which gives non-degenerate set-theoretic solutions of the Yang-Baxter equation that are not necessarily involutive.  Algebraic properties of skew left braces have been exploited to construct new solutions via matched products, semi-direct products, and asymmetric products \cite{DB18, BCJO18, BCJO19, CCS, CCS1, CCS2, WR08}. An extension theory for (skew) left braces has also been developed having expected connections with second cohomology of these objects \cite{DB18, CCS3, LV16, NMY1}. 
\par

Recently, in \cite{LHY2021}, Guo, Lang and Sheng introduced Rota-Baxter operators on Lie groups in connection with the well-studied Rota-Baxter operators on Lie algebras. Further study of Rota-Baxter operators on (abstract) groups has been carried out by Bardakov and Gubarev in \cite{VV2022, VV2023}, where they showed that every  Rota-Baxter operator on a group gives rise to a skew left brace structure on that group. Rota-Baxter operators on Clifford semigroups have been considered in \cite{CMP}. The idea of Rota-Baxter groups has been extended further by Jiang, Sheng and Zhu  \cite{JYC} to relative Rota-Baxter groups, which we pursue in this text.
\par

In this paper, we explore connections between relative Rota-Baxter groups and skew left braces. We prove that every  relative Rota-Baxter group gives rise to a skew left brace (Proposition \ref{rrb2sb}), and that every skew left brace arises from a relative Rota-Baxter group (Proposition \ref{skew left brace to rrbg}). It turns out that there is an isomorphism between the category of bijective relative Rota-Baxter groups and the category of skew left braces (Theorem \ref{iso brrbg and slb}). Using the result of	\cite[Proposition 3.4]{JYC}, we propose an efficient GAP algorithm, which would enable the computation of relative Rota-Baxter operators on finite groups. We introduce appropriate sub-objects and quotient objects in the category of relative Rota-Baxter groups and use them to define the notion of isoclinism for these objects, which generalises the notion of isoclinism of groups. We conclude by proving that an isoclinism of these objects induces an isoclinism of corresponding skew left braces (Theorem \ref{isoclinism rrbg implies isoclinism slb}).
\medskip

\section{Preliminaries on relative Rota-Baxter groups}
	
In this section, we recall some basic notions about relative Rota-Baxter groups that we shall need, and refer the readers to \cite{JYC,VV2022} for more details.

\begin{defn}
A relative Rota-Baxter group is a quadruple $(H, G, \phi, R)$, where $H$ and $G$ are groups, $\phi: G \rightarrow \Aut(H)$ a group homomorphism (where $\phi(g)$ is denoted by $\phi_g$) and $R: H \rightarrow G$ is a map satisfying the condition $$R(h_1) R(h_2)=R(h_1 \phi_{R(h_1)}(h_2))$$ for all $h_1, h_2 \in H$. 
\par
\noindent The map $R$ is referred as the relative Rota-Baxter operator on $H$.
\end{defn}

We say that the relative Rota-Baxter group  $(H, G, \phi, R)$ is trivial if $\phi:G \to \Aut(H)$ is the trivial homomorphism.

\begin{example}
Let $G$ be a group with subgroups $H$ and $L$ such that $G=HL$ and $H\cap L=\{1\}$. Then $(G,G, \phi, R)$ is a relative Rota-Baxter group, where $R: G \rightarrow G$ denotes the map given by $R(hl) = l^{-1}$ and $\phi : G \rightarrow \Aut(G)$ is the adjoint action, that is, $\phi_g(x) = gxg^{-1}$ for $g, h \in G$.
\end{example}

\begin{example}
Let $\mathbb{Z}_{2n}=\langle a \rangle$ and $\mathbb{Z}_{2m}=\langle b \rangle$ be cyclic groups of order $2n$ and $2m$, respectively. Let $R: \mathbb{Z}_{2m} \rightarrow \mathbb{Z}_{2n}$ be the map defined by 
$$
	R(b^k)=
\begin{cases}
a^n \quad \mbox{ if $k$ is an odd natural number, }  \\
1 \quad \mbox{ if $k$ is an even natural number.} 
\end{cases}
$$
Then $(\mathbb{Z}_{2m} , \mathbb{Z}_{2n}, \phi, R)$ is a relative Rota-Baxter group with respect to all homomorphisms $\phi: \mathbb{Z}_{2n} \rightarrow	\Aut(\mathbb{Z}_{2m})$. 
\end{example}

\begin{example}
Take $H= \mathbb{R}$ and $G=UP(2; \mathbb{R})$, the group of invertible upper triangular matrices. Let $\phi:UP(2; \mathbb{R})\to \Aut(\mathbb{R})$ be given by

$$\phi_{\begin{pmatrix} a & b \\ 0 & c \end{pmatrix}}(r)= ar$$
for $\begin{pmatrix} a & b \\ 0 & c \end{pmatrix} \in UP(2; \mathbb{R})$ and $r \in \mathbb{R}$. Further, let $R: \mathbb{R} \to  UP(2; \mathbb{R})$ be given by $$R(r)=\begin{pmatrix} 1 & r \\ 0 & 1 \end{pmatrix}.$$ Then $(\mathbb{R},UP(2; \mathbb{R}), \phi, R)$ is a relative Rota-Baxter group.
\end{example}

We now define a morphism of two relative Rota-Baxter groups with respect to the same action \cite[Definition 3.6]{JYC}.

\begin{defn}\label{morphism of rel rbo}
Let $(H, G, \phi, R)$ and $(H, G, \phi, S)$ be two relative Rota-Baxter groups. A morphism from $(H, G, \phi, R)$ to $(H, G, \phi, S)$ is a pair $(\psi, \eta)$, where $\psi: H \rightarrow H$ and   $\eta: G \rightarrow G $ are homomorphisms satisfying the conditions
$$\eta \; R= S \; \psi \quad \textrm{and} \quad	 \psi \; \phi_{g}= \phi_{\eta(g)} \; \psi$$
for all $g \in G.$

\end{defn}

\begin{defn}
A Rota-Baxter group is a group $G$ together with a map $R: G \rightarrow G$ such that
$$ R(x)  R(y)= R(x  R(x)  y  R(x)^{-1}) $$
for all $x, y \in G$. The map $R$ is referred as the Rota-Baxter operator on $G$.
\end{defn}

Let $\phi : G \rightarrow \Aut(G)$ be the adjoint action, that is, $\phi_g(x)=gxg^{-1}$ for $g, h \in G$. Then the relative Rota-Baxter group $(G, G, \phi, R)$ is simply a Rota-Baxter group.

\begin{prop}\label{R homo H to G} \cite[Proposition 3.5]{JYC} 
Let $(H, G, \phi, R)$ be a relative Rota-Baxter group. Then the operation 
\begin{align}
h_1 \circ_R h_2 := h_1 \phi_{R(h_1)}(h_2)
\end{align}
defines a group operation on $H$.  Moreover, the map $R: H^{(\circ_R)} \rightarrow G$ is a group homomorphism. The group $H^{(\circ_R)}$ is called the descendent group of $B$.
\end{prop}

\begin{remark}
If  $(H, G, \phi, R)$ is a relative Rota-Baxter group, then the image $R(H)$ of $H$ under $R$ is  a subgroup of $G$.
\end{remark}

\begin{thm}\label{SubgrptoRB}	\cite[Proposition 3.4]{JYC} Let $H$ and $G$ be groups and $\phi: G \rightarrow \Aut(H)$ be an action of $G$ on $H$. Then a map $R : H \rightarrow G$ is a relative Rota-Baxter operator if and only if the set 
\begin{align}\label{gr}
Gr(R)=\{(R(h), h) \mid h \in H \}
\end{align}
 is a subgroup of the semi-direct product $G \ltimes_{\phi} H$.
\end{thm}

\begin{remark}\label{iso H circ and Gr}
By Proposition \ref{R homo H to G},  $R: H^{(\circ_R)} \rightarrow G$ is a group homomorphism. Thus, it follows that $H^{(\circ_R)} \cong Gr(R)$ as groups via the the map $h \mapsto (R(h), h)$.
\end{remark}
\medskip

\section{Relative Rota-Baxter groups and skew left braces}
In this section, we explore relationship between relative Rota-Baxter groups and skew left braces.

\begin{defn}
A skew left brace is a  triple $(H,\cdot ,\circ)$, where $(H,\cdot)$ and $(H, \circ)$ are groups  such that
	$$a \circ (b \cdot c)=(a\circ b) \cdot a^{-1} \cdot (a \circ c)$$
	holds for all $a,b,c \in H$, where $a^{-1}$ denotes the inverse of $a$ in $(H, \cdot)$. The groups $(H,\cdot)$ and $(H, \circ)$ are called the additive and the multiplicative groups of the skew left brace $(H, \cdot, \circ)$, and will sometimes be denoted by $H ^{(\cdot)}$ and  $H ^{(\circ)}$, respectively.
\end{defn}

It is known that skew left braces give rise to non-degenerate solutions of the Yang-Baxter equation \cite[Proposition 1.9]{ GV17}.

\begin{prop}\label{lmap}	
Let $(H,\cdot ,\circ)$ be a skew left brace. The map $\lambda: H^{(\circ)} \rightarrow \operatorname{Aut}(H^{(\cdot)})$ defined by $\lambda_a(b) = a^{-1} \cdot (a \circ b)$ for $a,b \in H$, is a group homomorphism.  Furthermore, 
$$r_H: H \times H \rightarrow H \times H$$
given by
$$r_H(a,b)=\big(\lambda_a(b), \lambda^{-1}_{\lambda_a(b)}((a \circ b)^{-1} a (a \circ b)) \big)$$
is a non-degenerate solution of the Yang-Baxter equation.
\end{prop}

The map $\lambda$ defined above is referred as the associated $\lambda$-map of the skew left brace $(H,\cdot,\circ)$. Let us define some special types of skew left braces.
\begin{defn}
Let $(H, \cdot, \circ)$ be a skew left brace. Then
	\begin{enumerate}
		\item $(H, \cdot, \circ)$  is said to be a trivial skew left brace if $a \cdot b= a \circ b$ for all $a, b \in H$.
		\item $(H, \cdot, \circ)$ is called a bi-skew left brace if $(H,  \circ, \cdot)$ is also a skew left brace.
		
	\end{enumerate}
\end{defn}
\begin{defn}
	Let $(H, \cdot_H, \circ_H)$ and $(K, \cdot_K, \circ_K)$ be skew left braces. A map $\psi : H \rightarrow K$ is called a homomorphism of skew left braces if, for all $h_1, h_2 \in H$, it satisfies 
	$$\psi(h_1 \cdot_H h_2)=\psi(h_1) \cdot_K \psi(h_2) \quad \textrm{and} \quad \psi(h_1 \circ_H h_2)=\psi(h_1) \circ_K \psi(h_2).$$

\end{defn}
\begin{remark}
	Let $\lambda^H$ and $\lambda^K$ be $\lambda$-maps associated  to skew left braces $(H, \cdot_H, \circ_H)$ and $(K, \cdot_K, \circ_K)$, respectively. A group homomorphism $\psi:H^{(\cdot_H)} \rightarrow K^{(\cdot_K)}$ is homomorphism of corresponding skew left braces if and only, if for all $h \in H$, it satisfies the relation
	$$\psi \; \lambda^H_{h}= \lambda^K_{\psi(h)} \; \psi.$$
\end{remark}

\begin{prop}\label{rrb2sb}
Let $(H, G, \phi, R)$ be a relative Rota-Baxter group. If $\cdot$ denotes the  group operation of $H$, then the triple $(H, \cdot, \circ_R)$ is a skew left brace. 
\end{prop}

\begin{proof}
For $h_1, h_2, h_3 \in H$,  we have
\begin{align*}
(h_1 \circ_R h_2) \cdot h_1^{-1} \cdot (h_1 \circ_R h_3)=  & h_1 \cdot \phi_{R(h_1)}(h_2) \cdot (h_1^{-1} \cdot h_1) \cdot  \phi_{R(h_1)}(h_3) \\
= & h_1 \cdot (\phi_{R(h_1)}(h_2 \cdot h_3)) \\
=& h_1 \circ_R(h_2 \cdot h_3),
\end{align*}
which shows that $(H, \cdot, \circ_R)$ is a skew left brace.
\end{proof}

If $(H, G, \phi, R)$ is a relative Rota-Baxter group, then $(H, \cdot, \circ_R)$ is referred as the skew left brace induced by $R$ and will be denoted by $H_R$ for brevity.

\begin{remark}\label{iso H circ and Gr}
	Let $(H, G, \phi, R)$ be a relative Rota-Baxter group. By Proposition \ref{R homo H to G}, we conclude that there is an isomorphism between $H^{(\circ_R)}$ and $Gr(R)$ given by the map $h \mapsto (R(h), h)$. In other words, we can say that the graph of a relative Rota-Baxter operator is isomorphic to the multiplicative group of the skew left brace induced by it.
\end{remark}

\begin{remark}
It is worth noting that if $(H, G, \phi, R)$ is a relative Rota-Baxter group and $H_R$ its induced skew left brace, then the set theoretical solution of the Yang-Baxter equation defined by $H_R$ in \cite[Theorem 3.1]{GV17} is the same as the set-theoretical solution defined by the relative Rota-Baxter group $(H, G, \phi, R)$ as stated in \cite[Corollary 3.14]{BGST}.
\end{remark}

\begin{prop} Let $(H, G, \phi, R)$ be a relative Rota-Baxter group. Then $H_R$ is a trivial skew left brace if and only if $\im (R)\subseteq \Ker (\phi)$.
\end{prop}
\begin{proof}
The skew left brace $H_R$ is trivial if and only if, for all $ h_1, h_2 \in H_R$, we have
$$h_1 h_2 =h_1 \circ_R h_2=h_1 \phi_{R(h_1)}(h_2) \iff \phi_{R(h_1)}(h_2)=h_2.$$
This implies that $R(h) \in \Ker(\phi)$ for all $h \in H$.
\end{proof}

\begin{remark}
It follows from the preceding proposition that  $H_R$ cannot be trivial if $\phi$ is injective and $R$ is non-trivial. Furthermore, if $R$ is a bijection, then $H_R$ is a trivial skew left  brace. Since any relative Rota-Baxter operator induces a skew left brace, the following question seems natural.
\end{remark}

\begin{question}\label{p1}
Is every skew left brace induced by a relative Rota-Baxter operator ?
\end{question}

Let $(H,\cdot,\circ)$ be a skew left brace. It follows from Proposition \ref{lmap}  that the associated  $\lambda$-map $\lambda:H^{(\circ)} \to \Aut(H^{(\cdot)})$ is a group homomorphism. Thus, for any skew left brace $(H,\cdot,\circ)$, the group $H^{(\circ)}$ acts on $H^{(\cdot)}$ via the map $\lambda$. 
 \begin{prop}\label{skew left brace to rrbg}
Let $(H, \cdot, \circ)$ be a skew left brace and $\lambda : H^{(\circ)} \rightarrow \Aut (H^{(\cdot)})$ the associated $\lambda$-map. Then the quadruple $(H^{(\cdot)},  H^{(\circ)}, \lambda, \Id_H)$ is a relative Rota-Baxter group. Furthermore,  the skew left brace structure induced by the relative Rota-Baxter operator $\Id_H$  is the same as $(H, \cdot, \circ)$.
 \end{prop}

\begin{proof}
The proof follows from the definition of $\lambda$. 
\end{proof}

Let $(H, \cdot, \circ)$ be a skew left brace and $\lambda : H^{(\circ)} \rightarrow \Aut (H^{(\cdot)})$ its associated $\lambda$-map. Then the quadruple $(H^{(\cdot)}, H^{(\circ)}, \lambda, \mathrm{Id}_H)$ is referred as the relative Rota-Baxter group induced by the skew left brace $(H, \cdot, \circ)$. We now reformulate Theorem \ref{SubgrptoRB} in a manner suitable for computational purposes.

Before proceeding further, let us set some notations. Let $H$ and $G$ be groups, and $\phi: G \rightarrow \Aut(H)$ a homomorphism.
	\begin{enumerate}
		\item Let $\textbf{RB}(H, G, \phi)$ denote the set of all  maps $R:H \to G$ such that $(H, G, \phi, R)$ is a relative Rota-Baxter group.
		\item Let $\textbf{S}(H, G, \phi)$ denote the set of all subgroups $K$ of $G \ltimes_{\phi} H$ whose order equal to the order of $H$ and the natural projection from $K$ onto $H$ is surjective.
	\end{enumerate}

\begin{prop}\label{cal}
There is a one-to-one correspondence between  \textbf{RB}$(H, G, \phi )$ and $\textbf{S}(H, G, \phi)$.
\end{prop}
\begin{proof}
	Define $\psi: \textbf{RB}(H, G, \phi ) \rightarrow \textbf{S}(H, G, \phi)$  by
	$$\psi(R)= Gr(R)$$
for $R \in  \textbf{RB}(H, G, \phi )$. The map $\psi$ is well-defined by Theorem \ref{SubgrptoRB}.  Next, we define the inverse of $\psi$. If $M \in\textbf{S}(H, G, \phi)$, then by definition of $\textbf{S}(H, G, \phi)$, for every $h \in H$ there exists a unique $g_{h} \in G$ such that $(g_{h}, h)  \in M$. Define $R_M : H \rightarrow G$ by $R_M(h):=g_{h}$. For $h_1, h_2 \in H$, we have $$(g_{h_1}, h_1) (g_{h_2}, h_2) = (g_{h_1} g_{h_2}, h_1 \phi_{g_{h_1}}(h_2)).$$ Since $M$ is a group, $(g_{h_1}, h_1) (g_{h_2}, h_2) \in M$, which implies that $g_{h_1 \phi_{g_{h_1}} (h_2)} = g_{h_1} g_{h_2}$. Thus, we conclude that $(H, G, \phi, R_M)$ is a relative Rota-Baxter group. Now, we define a map $\eta: \textbf{S}(H, G, \phi) \to \textbf{RB}(H, G, \phi)$ by setting $$\eta(M) = R_M$$ for $M \in \textbf{S}(H, G, \phi)$. 	It is easy to see that $\psi$ and $\eta$ are inverses of each other, which proves the proposition.

\end{proof}

Suppose that $H$ is a  finite group and $G$ a group with the same order as that of $H$. A fundamental problem in the theory of skew left braces is whether there exists a skew left brace structure on $H$ such that its multiplicative group is isomorphic to $G$. As a consequence of Proposition \ref{cal}, we provide a necessary and sufficient condition for the existence of a skew left brace structure in terms of subgroups of semi-direct products of $G$ by $H$.

\begin{cor}
	Let $H$ be a finite group and $G$ a group with the same order as $H$. Then, there exists a skew left brace structure on $H$ whose multiplicative group is isomorphic to $G$ if and only if there exists $M \in \textbf{S}(H, G, \phi)$ for some action $\phi: G \rightarrow \Aut(H)$, such that $M$ is isomorphic to $G$.
\end{cor}

\begin{proof}
	Suppose that there exists $M \in \textbf{S}(H, G, \phi)$ for some action $\phi: G \rightarrow \Aut(H)$, such that $M$ is isomorphic to $G$. By Remark \ref{iso H circ and Gr}, we know that $M$ is isomorphic to the multiplicative group of the skew left brace induced by $R_M$. Since $M$ is isomorphic to $G$, it follows that the skew left brace induced by $R_M$ has a multiplicative group isomorphic to $G$.
\par
Conversely, suppose that there exists a skew left brace structure $(H, \cdot, \circ)$ such that $H^{(\circ)}$ is isomorphic to $G$. Let $\lambda: H^{(\circ)} \rightarrow \Aut(H^{(\cdot)})$ be the $\lambda$-map associated to the skew left brace. By Proposition \ref{skew left brace to rrbg}, we know that $(H^{(\cdot)}, H^{(\circ)}, \lambda, \Id_H)$ is a relative Rota-Baxter group. Further, it follows that $Gr(\Id_H)$ is an element of $\textbf{S}(H^{(\cdot)}, H^{(\circ)} , \lambda)$ and is isomorphic to $H^{(\circ)}$.
\end{proof}

Let $M, N \in \textbf{S}(H, G, \phi)$ such that $M \cong N$. In view of Proposition \ref{cal}, the groups $M$ and $N$ correspond to the  graphs of $R_M$ and $R_N$, respectively. Since $M\cong N$, it follows that $H^{(\circ_{R_M})} \cong H^{(\circ_{R_N})}$. This observation leads to the following natural question.

\begin{question}\label{imageiso}
	Let $(H,G, \phi,R)$ and $(H,G, \phi,S)$ be two relative Rota-Baxter groups such that $Gr(R) \cong  Gr(S)$. Can we conclude that the groups $R(H)$ and $S(H)$ are isomorphic?
\end{question}

In general, the answer to Question \ref{imageiso} is not always positive. For instance, take $H=G=\mathbb{Z}_5$ and $\phi: G \rightarrow \Aut(H)$ the trivial action. In this case, the groups $M= \langle (1,1)\rangle$ and $N= \{0\} \times \mathbb{Z}_5$ are isomorphic and both lie in $\textbf{S}(H,G,\phi)$. The induced relative Rota-Baxter operators $R_M$ and $R_N$ are the identity map and the trivial homomorphism, respectively, and hence $R_M(H)=\mathbb{Z}_5$ and $R_N(H)=\{0\}$. Thus, $R_M(H) \not\cong R_N(H)$.

We are interested in isomorphism classes of skew left braces induced by relative Rota-Baxter groups. Let $H_R$ and $H_B$ be skew left braces induced by relative Rota-Baxter groups 
$(H, G, \phi, R)$ and $(H, G, \phi, B)$, respectively. Let `$\cdot$' denote the group structure on $H$. Then $H_R \cong H_B$ if and only if there exists a group isomorphism $\psi_H: H^{(\cdot)} \rightarrow H^{(\cdot)}$ such that
\begin{align}\label{isorb}
 \psi_H \; \phi_{R(h)}= \phi_{B(\psi_H(h))} \; \psi_H
\end{align}
 holds for all $h \in H$.
 
It follows from Definition \ref{morphism of rel rbo} that if $(\psi_H, \psi_G): (H, G, \phi, R) \to (H, G, \phi, B)$ is a morphism, then $\psi_H: H_R \to H_B$ is a homomorphism of induced skew left braces. Hence, isomorphic  relative Rota-Baxter groups induce isomorphic skew left braces. This answers a question of Bardakov and Gubarev \cite[p. 21]{VV2022} as follows.

\begin{cor}
Two Rota-Baxter operators $R$ and $B$ on a group $H$  induce isomorphic skew left braces if and only if there exists an isomorphism $\psi_H : H \rightarrow H $ such that
\begin{align}\label{RBisom}
\psi_H(R(h))^{-1}  B(\psi_H(h)) \in \Z(H) \quad \textrm{for all} \quad h \in H.
\end{align}
In particular, if $H$ has trivial center, then \eqref{RBisom} can be reduced to
\begin{align}\label{clessg}
\psi_H \; R=B \;  \psi_H.
\end{align}
\end{cor}

\begin{proof}
A Rota-Baxter operator is a special case of a relative Rota-Baxter operator with the action $\psi:H \to \Aut(H)$ being the adjoint action. The result now follows immediately from \eqref{isorb}.
\end{proof}

Note that every skew left brace structure on a complete group is induced by some Rota-Baxter operator on that group \cite[Proposition 3.12]{VV2022}. The skew brace structures for groups of order $96$ are currently unknown \cite{VVY022}. We use Theorem \ref{cal} to count the total number of Rota-Baxter operators on centerless groups of order $96$. Condition \eqref{clessg} is very useful  to count equivalence classes of Rota-Baxter operators on small order groups since it is independent of the action $\phi$.

\begin{Algo}
Let $G$ and $H$  be finite groups and $\phi: G \rightarrow \Aut(H)$ an action. Let $S:= G \ltimes_{\phi} H$, $\mathcal{M}$ be the set of all subgroups of $S$ and $\mathcal{N}$ be the subset of $\mathcal{M}$ consisting of those subgroups whose order is equal to order of $H$. We want to determine elements $A \in \mathcal{N}$ such that the natural projection $ A \rightarrow H$ onto the second coordinate  is onto. But, there is no natural way to define such a map, since $H$ is not stored in the GAP library in the form of tuples. So, we first set some maps to make the task easy.

\begin{enumerate}
\item $E_1:= Embedding(S,1)$, this gives embedding of  $G$ inside $S$.
\item $E_2:= Embedding(S,2)$, this gives an embedding of  $H$ inside $S$.
\item $p:=Projection(S)$, this gives projection of $S$ onto $G$. 
\item Define $C: S \rightarrow E_2(H)$ by $C(x)=x \; E_1(p(x)^{-1})$.  The map $C$ will work as a natural projection map.
\end{enumerate}
Now a subgroup $A \in \textbf{S}(H, G, \phi)$ if and only if $C|_{A}: A \rightarrow E_2(H)$ is onto. In other words,  $A \in \textbf{S}(H, G, \phi)$ when the cardinality of $C|_{A}(A)$ is the same as that of $H$. Hence,  the number of relative Rota-Baxter operators from  $H$ to $G$ with respect to the action $\phi$ is equal to the number of elements of $\mathcal{N}$ for which the restriction of the map $C$ is a bijection.

Now assume that $G=H$ and $\phi$ is the adjoint action. In this case, we give an algorithm to explicitly define all Rota-Baxter operators on $G$. 

Let $C_A:=C|_{A}$ for $A \in \textbf{S}(G, G, \phi)$. Define $R_A: G \rightarrow G$ by
$$R_A(x):= p(C^{-1}_A(E_2(x) )).$$
Then the map $R_A$ is a Rota-Baxter operator on $H$. Next, define a relation $\sim$  on $\textbf{S}(G, G, \phi)$ by the following rule. For $A, B \in \textbf{S}(G, G, \phi)$, we say $A \sim B$ if there exists an automorphism $\psi_H : H \rightarrow H$ such that 
$$ \psi_H(R_A(h))^{-1}  R_B(\psi_H(h)) \in \Z(H) \quad \textrm{for all} \quad h \in H. $$   

The number of $\sim$ equivalence classes gives a lower bound on the number of skew left brace structures on $H$. Further,  if $\Z(H)$ is trivial, then the number of $\sim$ equivalence classes  is precisely equal to the number of skew left brace structures on $H$.
\end{Algo}

Let  $|RBO|$  denote the total number of Rota-Baxter operators on a given group. Using GAP, we have discovered that there are  5 centerless groups of order  96.
\medskip
\begin{center}
\begin{tabular}{|l|*{6}{c|}r}
	\hline
	GAP Group Id              & (96,64) & (96,70) & (96, 71) & (96, 72) & (96, 227)  \\
	\hline
$|RBO|$               & 352 & 1512 & 528 & 552 & 4504\\ 
	\hline
\end{tabular}
\end{center}
\medskip

We have shown  that every relative Rota-Baxter group  $(H,G, \phi, R)$ can be identified by an element in the set $\textbf{S}(H,G,\phi)$. It is well-known that non-isomorphic skew left braces can have isomorphic additive and multiplicative groups. Hence, if we have $M, N \in \textbf{S}(H,G,\phi)$ such that $M\cong N$, we cannot assume that the corresponding skew left braces induced by $M$ and $N$ are isomorphic. Consequently, the relative Rota-Baxter groups $(H, G, \phi, R_M)$ and $(H, G, \phi, R_N)$ induced by $M$ and $N$, respectively, may not be isomorphic. This poses the following problem:

\begin{Prob}
Given $M, N \in \textbf{S}(H, G, \phi)$, under what conditions are the skew left braces induced by $R_M$ and $R_N$ are isomorphic? In stronger terms, under what conditions are the relative Rota-Baxter groups $(H, G, \phi, R_M)$ and $(H, G, \phi, R_N)$ isomorphic?
\end{Prob}

It is well-known that, a skew left brace structure on a group $H$ can be identified with a regular subgroup of the Holomorph  $\Hol(A)$ of $H$. Furthermore, if the corresponding subgroup is normal, the resulting skew left brace is a bi-skew left brace \cite[Theorem 3.4]{AC20}. A similar situation arises when examining a skew left brace structure on $H$ via a group lying in $\textbf{S}(H, G, \phi)$.

\begin{Prob}
What can we say about the skew left brace $H_M$ induced by $R_M$, when $M$ is a normal or a characteristic subgroup of $G \ltimes_{\phi} H$ lying in $\textbf{S}(H, G, \phi)$?
\end{Prob} 
\medskip

\section{Morphisms of relative Rota-Baxter groups}
We now provide a broader definition of morphism of relative Rota-Baxter groups since the original definition is applicable only to relative Rota-Baxter groups with respect to the same action.

	Let  $(H, G, \phi, R)$ be a relative Rota-Baxter group, and let $K \leq H$ and $L \leq G$ be subgroups.
	\begin{enumerate}
\item If $K$ is $L$-invariant under the action $\phi$, then we denote the restriction of $\phi$ by $\phi|: L \to \Aut(K)$. 
\item If $R(K) \subseteq L$, then we denote the restriction of $R$ by $R|: K \to L$.
\end{enumerate}

\begin{defn}
	Let $(H,G,\phi,R)$ be a relative Rota-Baxter group, and $K\leq H$ and $L\leq G$ be subgroups. Suppose that  $\phi_\ell(K) \subseteq K$ for all $\ell \in L$ and $R(K) \subseteq L$. Then $(K,L,\phi |,R |)$ is a relative Rota-Baxter group, which we refer as a relative Rota-Baxter subgroup of $(H,G,\phi,R)$ and write $(K,L,\phi |,R |)\leq(H,G,\phi,R)$.
\end{defn}

\begin{defn}
Let $(H, G, \phi, R)$ and $(K, L, \varphi, S)$ be two relative Rota-Baxter groups.
\begin{enumerate}
\item A homomorphism $(\psi, \eta): (H, G, \phi, R) \to (K, L, \varphi, S)$ of relative Rota-Baxter groups is a pair $(\psi, \eta)$, where $\psi: H \rightarrow K$ and $\eta: G \rightarrow L$ are group homomorphisms such that
\begin{equation}\label{rbb datum morphism}
\eta \; R = S \; \psi \quad \textrm{and} \quad \psi \; \phi_g =  \varphi_{\eta(g)} \; \psi \quad \textrm{for all} \quad g \in G.
\end{equation}

\item The kernel of a homomorphism $(\psi, \eta): (H, G, \phi, R) \to (K, L, \varphi, S)$ of relative Rota-Baxter groups is the quadruple $$(\Ker(\psi), \Ker(\eta), \phi|, R|),$$ where $\Ker(\psi)$ and $\Ker(\eta)$ denote the kernels of the group homomorphisms $\psi$ and $\eta$, respectively. The conditions in \eqref{rbb datum morphism} imply that the kernel is itself a relative Rota-Baxter group.

\item The image of a homomorphism $(\psi, \eta):  (H, G, \phi, R) \to (K, L, \varphi, S)$ of relative Rota-Baxter groups is the quadruple 
$$(\IM(\psi), \IM(\eta), \varphi|, S| ),$$ where $\IM(\psi)$ and $\IM(\eta)$ denote the images of the group homomorphisms $\psi$ and $\eta$, respectively. The image is itself a relative Rota-Baxter group.

\item A homomorphism $(\psi, \eta)$ of relative Rota-Baxter groups is called an isomorphism if both $\psi$ and $\eta$ are group isomorphisms. Similarly, we say that $(\psi, \eta)$ is an embedding of a relative Rota-Baxter group if both $\psi$ and $\eta$ are embeddings of groups.
\end{enumerate}
\end{defn}
\begin{remark}
Clearly, if $(\psi, \eta): (H, G, \phi, R) \to (K, L, \varphi, S)$ is a homomorphism of relative Rota-Baxter groups, then $(\Ker(\psi), \Ker(\eta), \phi|, R|)$ is a relative Rota-Baxter subgroup  of $(H,G,\phi,R)$ and  $(\IM(\psi), \IM(\eta), \varphi|, S|)$ is a relative Rota-Baxter subgroup of $(K, L, \varphi, S)$.
\end{remark}

\begin{prop}\label{rrb to slb homo}
A homomorphism of relative Rota-Baxter groups induces a homomorphism of corresponding skew left braces. 
\end{prop}

\begin{proof}
Let $(\psi, \eta): (H, G, \phi, R) \to (K, L, \varphi, S)$ be a homomorphism of relative Rota-Baxter groups. Let $H_R$ and $K_S$ be induced skew  braces. Then, for $x, y \in H$, we have
$$\psi(x \circ_R y)=\psi(x \phi_{R(x)}(y))=\psi(x) \psi(\phi_{R(x)}(y))= \psi(x)  \varphi_{\eta(R(x))}\psi(y)=\psi(x)  \varphi_{S (\psi(x))}\psi(y)=\psi(x) \circ_{S}\psi(y),$$
and hence $\psi: H_R \to K_S$ is a homomorphism of induced skew left braces. 
\end{proof}

Suppose that $(H, \cdot, \circ)$ is a skew left brace and $\lambda : H^{(\circ)} \rightarrow \Aut (H^{(\cdot)})$ the associated $\lambda$-map. Then  the quadruple $(H^{(\cdot)}, H^{(\circ)}, \lambda, \mathrm{Id}_H)$ is called the relative Rota-Baxter group  induced by the skew left brace $(H, \cdot, \circ)$.

\begin{prop}\label{slb to rrb homo} 
A homomorphism of skew left braces induces a homomorphism of corresponding relative Rota-Baxter groups.
\end{prop}

\begin{proof}
Let $\psi: (H, \cdot_H, \circ_H) \to (K, \cdot_K,  \circ_K)$ be a morphism of skew left braces. Then $\psi$ can be viewed as a homomorphism $H^{(\cdot_H)} \to K^{(\cdot_K)}$ and $H^{(\circ_H)} \to K^{(\circ_K)}$. Further, we see that
$\psi \; \Id_H = \Id_K \; \psi$ and $\psi \; \lambda^H_h =  \lambda^K_{\psi(h)} \; \psi$ for all $h \in H$, where $\lambda^H$ and $\lambda^K$ are the $\lambda$-maps of $(H, \cdot_H, \circ_H)$ and $(K, \cdot_K, \circ_K)$, respectively. Hence, $(\psi, \psi): (H^{(\cdot_H)}, H^{(\circ_H)}, \lambda^H, \Id_H ) \to (K^{(\cdot_K)}, K^{(\circ_K)}, \lambda^K, \Id_K )$ is a homomorphism of relative Rota-Baxter groups.
\end{proof}

\begin{prop}\label{bihom}
Let $(H, G, \phi, R)$ and $(K, L, \varphi, S)$ be two relative Rota-Baxter groups such that $R$ and $S$ are bijections. Then, there is a one-to-one correspondence between the set of homomorphisms from $H_R$ to $K_S$ and the set of homomorphisms from $(H, G, \phi, R)$ to $(K, L, \varphi, S)$.
\end{prop}

\begin{proof}
	Let $\psi: H_R \rightarrow K_S$ be a homomorphism of skew left braces. Then, for all $h \in H$, we have
\begin{equation}\label{bhom0}
	\psi \; \phi_{R(h)} =\varphi_{S(\psi(h))} \; \psi.
\end{equation}
	 Define $\eta: G \rightarrow L$, by $\eta(g)=S(\psi(h_g))$, where $g=R(h_g)$ for some unique $h_g \in H$. The map $\eta$ is well-defined since $R$ and $S$ are bijections. For $g_1, g_2 \in G$, we have 
	\begin{equation}\label{bhom}
		\eta(g_1 g_2)=S(\psi(h_{g_1 g_2})),
	\end{equation}
	where $R(h_{g_1 g_2})= g_1 g_2$. Given that $g_1= R(h_{g_1})$ and $g_2=R(h_{g_2})$, and since $R$ is a relative Rota-Baxter operator, we can write  $h_{g_1} \phi_{R(h_{g_1})}(h_{g_2})=h_{g_1} \circ_R h_{g_2}$. By substituting this value in \eqref{bhom}, noting that $\psi$ is a homomorphism of skew left braces and $S$ is a relative Rota-baxter operator, we obtain
	\begin{align*}
		\eta(g_1 g_2)=\; & S(\psi(h_{g_1} \phi_{R(h_{g_1})}(h_{g_2})))\\
		=\; &S(\psi(h_{g_1}) \varphi_{S(\psi(h_{g_1}))}(\psi(h_{g_2})))\\
		=\;&  S(\psi(h_{g_1})) S(\psi(h_{g_2}))\\
		=\; & \eta(g_1) \eta(g_2).
	\end{align*}
	This shows that $\eta$ is a group homomorphism. By definition of $\eta$, we have $\eta \; R =S\; \psi$. Further, \eqref{bhom0} implies that, for each $g$ in $G$, we have
$$	\psi \; \phi_{g}= \varphi_{\eta(g)} \; \psi.
$$
	Hence, $(\psi, \eta):(H, G, \phi, R) \to (K, L, \varphi, S)$ is a homomorphism of relative Rota-Baxter groups. It follows immediately  that the map $\psi \mapsto (\psi, \eta)$ is a bijection from $\Hom(H_R, K_{S})$ to $\Hom ((H, G, \phi, R), (K, L, \varphi, S))$.

\end{proof}

We say that a  relative Rota-Baxter group $(H, G, \phi, R)$ is bijective if the Rota-Baxter operator $R:H \to G$ is a bijection. 

\begin{thm}\label{iso brrbg and slb}
There is an isomorphism between the category $\mathcal{BRRB}$ of bijective relative Rota-Baxter groups and the category $\mathcal{SLB}$ of skew left braces.
\end{thm}

\begin{proof}
Define $\mathcal{F} :\mathcal{BRRB} \rightarrow \mathcal{SLB}$ by $\mathcal{F}(H,G,\phi,R) = H_R$ and $\mathcal{G}:\mathcal{SLB} \rightarrow \mathcal{BRRB}$ by $\mathcal{G} (H,\cdot,\circ)=(H^{(\cdot)},H^{(\circ)},\lambda,\text{Id}_H)$. Using Propositions \ref{rrb to slb homo} and \ref{slb to rrb homo}, one can show that $\mathcal{F}$ and $\mathcal{G}$ are functors.
\par

Let $(H, \cdot, \circ)$ be a skew left brace. Since the skew left brace induced by the relative Rota-Baxter group $(H^{(\cdot)},H^{(\circ)},\lambda,\Id_H)$ is the same as $(H, \cdot, \circ)$, it follows that $\mathcal{F}\circ \mathcal{G}$ is the identity functor.
\par

Let $(H,G,\phi,R)$ be a relative Rota-Baxter group. Then the relative Rota-Baxter group induced by $H_R$ is $(H^{(\cdot)}, H^{(\circ_R)}, \lambda, \mathrm{Id}_H)$, where $\lambda: H^{(\circ_R)} \rightarrow \mathrm{Aut}(H^{(\cdot)})$ is defined by $\lambda_h=\phi_{R(h)}$. Since $R:H \to G$ is a bijection, it follows that $R:H^{(\circ_R)} \to G$ is an isomorphism of groups. The facts that $R\; \mathrm{Id}_H=R \; \mathrm{Id}_H$ and $\mathrm{Id}_H \; \lambda_h=\phi_{R(h)}\; \mathrm{Id}_H$ immediately establishes that $(\mathrm{Id}_H, R): (H^{(\cdot)}, H^{(\circ_R)}, \lambda, \mathrm{Id}_H) \rightarrow (H,G,\phi,R)$ is an isomorphism of relative Rota-Baxter groups. 
\end{proof}
\medskip

%Let $H$ be a group and $N \normal H$. Suppose $R: N \rightarrow G$ be a relative Rota-Baxter operator with respect to an action $(N,\phi)$. Further assume that  $\phi_g$ can be lifted to $H$ for every $g \in G$. 

%\begin{defn}
%Let $H, G$ be two groups and $\phi: G \rightarrow \Aut(H)$ is a homomorphism. A  \textbf{Matching Relative Rota-Baxter operator}  on $G$ with respect to an action $(H,\phi)$ is a collection of maps   $\{R_x : H \rightarrow G\}_{x \in X}$ satisfying
%\begin{align}\label{MRRB}
%R_x(h_1) R_y(h_2)=R_x(h_1 \phi_{R_y(h_1)}(h_2))  \hspace{.4cm} \mbox{for all} \hspace{.2cm} h_1, h_2 \in G \hspace{.2cm} \mbox{and} \hspace{.2cm}x,y \in X.
%\end{align}
%\end{defn}

%\begin{remark}
%Taking $x=y$ in \eqref{MRRB} inplies that $R_x$ is a relative Rota-Baxter operator on $G$ with respect to action $(H, \phi)$, for all $x \in X$.
%\end{remark}

\section{Substructures in the category of relative Rota-Baxter groups}
We introduce substructures in the category of relative Rota-Baxter groups, and use them to give a comprehensive definition of quotients in this category. 
\par

Let $(H, G, \phi, R)$ be a relative Rota-baxter group, and consider subgroups $K \leq H$ and $L \leq G$.
 We denote the quadruple $(H, \im(R), \phi|, R|)$ by $\I(H, G, \phi, R)$, which is clearly a relative Rota-Baxter group.

\begin{prop}\label{image same skew brace}
Let $(H, G, \phi, R)$ be a relative Rota-Baxter group. Then  the skew left braces induced by $(H, G, \phi, R)$ and $\I(H, G, \phi, R)$ are the same.
\end{prop}

\begin{proof}
The proof follows from the definition of an induced skew left brace.
\end{proof}

\begin{defn}\label{defn ideal rbb-datum}
Let $(H, G, \phi, R)$ be a relative Rota-Baxter group and  $(K, L,  \phi|, R|) \leq (H, G, \phi, R)$ its relative Rota-Baxter subgroup. We say that $(K, L,  \phi|, R|)$ is an ideal of $(H, G, \phi, R)$ if 
\begin{align}
& K \trianglelefteq H \quad \mbox{and} \quad L \trianglelefteq G, \label{I0}\\
& \phi_g(K) \subseteq K  \mbox{ for all } g \in G, \label{I1} \\
& \phi_\ell(h) h^{-1} \in K \mbox{ for all } h \in H \mbox{ and }  \ell \in L. \label{I2}
\end{align}
We write $(K, L, \phi|, R|) \trianglelefteq (H, G, \phi, R)$ to denote an ideal of a relative Rota-Baxter group. 
\end{defn}

Next, we introduce quotient of a relative Rota-Baxter group.

\begin{thm}\label{subs}
 Let $(H, G, \phi, R)$ be a relative Rota-Baxter group and $(K, L,  \phi|, R|)$ an ideal of $(H, G, \phi, R)$. Then there are maps $\overline{\phi}: G/L \to \Aut(H/K)$ and  $\overline{R}: H/K \to G/L$ defined by
$$ \overline{\phi}_{\overline{g}}(\overline{h})=\overline{\phi_{g}(h)} \quad \textrm{and} \quad	\overline{R}(\overline{h})=\overline{R(h)}$$
 for $\overline{g} \in G/L$ and $\overline{h} \in H/K$, such that  $(H/K, G/L, \overline{\phi}, \overline{R})$ is a relative Rota-Baxter group.
\end{thm}

\begin{proof}
The conditions (1)--(4) are tailor-made for the maps  $\overline{\phi}$ and  $\overline{R}$ to be well-defined. Further, the map $\overline{R}$ satisfies the relative Rota-Baxter identity since $R$ does so. In order to prove that  $(K, L, \phi|, R|)$  is a relative Rota-Baxter group, it suffices to show that $\phi_\ell$ restricted to $K$ is an automorphism of $K$ for all $\ell \in L$. Since $\phi_\ell$ is already injective for all $\ell \in L$, it remains to show that $\phi_\ell$ restricted to $K$ is surjective. If $k \in K$, then there exists a $h \in H$ such that 
  \begin{align}\label{first}
  \phi_\ell(h)=k.
  \end{align}
Using  condition (2), we can write
   \begin{align}\label{sec}
   \phi_\ell(h)=k' h
\end{align}
for some $k' \in K.$ Comparing \eqref{first} and \eqref{sec} shows that $h \in K$, which is desired.
\end{proof}

\begin{no}
We write $(H, G, \phi, R)/(K, L, \phi|, R|)$ to denote the relative Rota-Baxter group $(H/K, G/L, \overline{\phi}, \overline{R})$.
\end{no}

\begin{defn}
 A sub skew left brace $(I, \cdot, \circ)$ of a skew left brace $(H,\cdot ,\circ)$ is called an ideal of $(H,\cdot ,\circ)$ if $I^{(\cdot)}$ is a normal subgroup of $H^{(\cdot)}$, $I^{(\circ)}$ is a normal subgroup of $H^{(\circ)}$ and $\lambda_a(I) \subseteq I$ for all $a \in H$.
\end{defn}

\begin{prop}\label{IRRB}
If $(K, L, \phi|,R|)$ is  an ideal of $(H, G, \phi, R)$, then $K_{R|}$ is an ideal of $H_R$.
\end{prop}

\begin{proof}
Let $H_R:= (H, \cdot, \circ_R)$ be the skew left brace induced by $R$, as defined in  Proposition \ref{rrb2sb}. Since $(K, L, \phi|,R|)$ is  an ideal of $(H, G, \phi, R)$,  by \eqref{I0} and \eqref{I1}, we have $K^{( \cdot)}$ is a normal subgroup of $H^{( \cdot)}$ and $\phi_{R(h)}(K) \subseteq K$ for all $h \in H$. It remains to show that $K^{(\circ_R)}$ is a normal subgroup of $H^{(\circ_R)}$. If $k \in K$, $h \in H$ and $h^{\dagger}$ denotes the inverse of $h$ in $H^{(\circ_R)}$, then we have 
\begin{align*}
h \circ_R k \circ_R h^{\dagger} = \; & h\circ_R (k \cdot \phi_{R(k)}(h^\dagger))\notag \\
= \; & h \cdot \phi_{R(h)}(k \cdot \phi_{R(k)}(h^\dagger))\notag \\
=\; & h \cdot \phi_{R(h)}(k \cdot \phi_{R(k)}(\phi_{R(h)^{-1}}(h)^{-1})) \notag \\
=  \; &  h \cdot \phi_{R(h)}(k) \cdot \phi_{R(h) R(k)R(h)^{-1}}(h^{-1})\notag \\
= \; & (h \cdot \phi_{R(h)}(k) \cdot h^{-1} )\cdot (h \cdot \phi_{R(h) R(k)R(h)^{-1}}(h^{-1}).\notag
\end{align*}
Since $K^{(\cdot)}$ is a normal subgroup of $H^{(\cdot)}$ and  $\phi_{R(h)}(K) \subseteq K$ for all $h \in H$,  it follows that $h \cdot \phi_{R(h)}(k) \cdot h^{-1} \in K$. Further, using the fact $R(K) \subseteq L$ and condition \eqref{I2}, we obtain $h \cdot \phi_{R(h) R(k)R(h)^{-1}}(h^{-1}) \in K$. This shows that $K$ is an ideal of the skew left brace $H_R$.
\end{proof}

%\begin{remark}
%In general, the converse of the Proposition \ref{IRRB}, stated above may not be true. However, we do have the following proposition.
%\end{remark}
%\begin{prop}\label{IRRB2}
%Let $(H, G, \phi, R)$ be a RRB-daum and  $K$ be an ideal of $H_R$. Assume that $L$ is normal  subgroup generated by image of $K$ under $R$. Then  $(K, L, \phi^{L,K},R^{L}_K)$ is  an ideal of a RRB-datum $\I (H, G, \phi, R)$ if $L =\im(R)$.
%\end{prop}

\begin{prop}
Let $(K, L, \phi|, R|)$ be  an ideal of $(H, G, \phi, R)$. Then the skew left brace induced by $(H/K, G/L, \overline{\phi}, \overline{R})$  is isomorphic to $ H_R/ K_{R|}$.
\end{prop}

\begin{proof}
It is apparent that the identity map serves as the necessary isomorphism.
\end{proof}
\medskip

\section{Isoclinism of relative Rota-Baxter groups}
In this section, we introduce the notion of isoclinism of relative Rota-Baxter groups, and relate it with the recently introduced notion of isoclinism of skew left braces \cite{TV23}. The idea stems from the classical work of Hall \cite{MR0003389} where isoclinism of groups was studied. To proceed, we need to introduce some relevant definitions in the context of relative Rota-Baxter groups.
\par
Let $(H, G, \phi, R)$ be a relative Rota-Baxter group. Then, by Proposition \ref{R homo H to G}, $R: (H, \circ_R) \rightarrow G$ is a group homomorphism, and hence $\Ker(\phi ~ R)$ is well-defined. We set
$$\Z^{\phi}_{R}(H):=\Z(H)  \cap \Ker(\phi ~ R) \cap \Fix(\phi),$$
where $\Z(H)$ is the center of the group $H$ and $\Fix(\phi)= \{ x \in H  \mid \phi_g(x)=x \mbox{ for all } g \in G \}$ is the fixed-point subgroup of the action.

\begin{defn}
The center of a relative Rota-Baxter group $(H, G, \phi, R)$ is defined as
	$$\Z(H, G, \phi, R):= \big( \Z^{\phi}_R(H), \Ker(\phi), \phi|, R|).$$
	\end{defn}

Let $(H, \cdot, \circ)$ be a skew left brace, $\Z(H^{(\cdot)})$ denote the centre of the group $H^{(\cdot)}$ and $\Fix(\lambda) = \{x \in H \mid \lambda_a(x) = x \quad \textrm{for all} \quad  a\in H \}$.  Then the annihilator of $(H,\cdot ,\circ)$ is defined as $$\Ann(H):=\Ker(\lambda) \cap \Z(H^{(\cdot)}) \cap \Fix(\lambda)=\{a \in H \mid b \circ a = a \circ b = b \cdot a = a \cdot b \quad \textrm{for all} \quad b \in H\}.$$ Clearly, $\Ann(H)$ is an ideal of  $(H, \cdot, \circ)$.

\begin{prop}\label{lem:mod-iso}
Let $(H, G, \phi, R)$ be a relative Rota-Baxter group. Then the following hold:
\begin{enumerate}
\item  $\Z(H, G, \phi, R)$ is an ideal of $(H, G, \phi, R)$ and the skew left brace induced by  $\Z(H, G, \phi, R)$  is trivial.
\item The skew left brace induced by  $\Z(\I(H, G, \phi, R))$ is the same as $\Ann(H_{R|})$.
\item The skew left brace induced by $\I(H, G, \phi, R)/ \Z(\I(H, G, \phi, R))$ is the same as $H_{R|}/ \Ann(H_{R|})$.
\end{enumerate}
\end{prop}

\begin{proof}
We prove each assertion individually.

\begin{enumerate}
\item It is clear that $\Z(H, G, \phi, R)$ is a relative Rota-Baxter subgroup of  $(H, G, \phi, R)$. To show that it is an ideal of $(H, G, \phi, R)$, we verify conditions \eqref{I0}, \eqref{I1} and \eqref{I2}. Clearly,  $\Z^{\phi}_R(H)  \trianglelefteq H$ and $\Ker(\phi) \trianglelefteq G$. If $g \in G$ and $x \in \Z^{\phi}_{R}(H)$, then $\phi_g(x)=x$, and hence $\phi_g (\Z^{\phi}_{R}(H))=\Z^{\phi}_{R}(H)$. If $\ell \in \Ker(\phi)$ and $h \in H$, then $\phi_\ell(h)h^{-1}=1$. Finally, the skew left brace induced by  $\Z(H, G, \phi, R)$  is trivial since $x \circ_{R|} y = x \cdot y$ for all $x, y \in \Z^{\phi}_{R}(H)$. This established assertion (1).

\item It suffices to prove that $\Z^{\phi|}_{R|}(H) = \Ann(H_{R|})$ as sets. Recall that $\Ann(H_{R|})= \Z(H^{(\cdot)}) \cap \Ker(\lambda) \cap \Fix (\lambda)$, where $\lambda: H^{(\circ_R)} \to \Aut(H^{(\cdot)})$ is the $\lambda$-map associated to the skew left brace $H_{R|}$. But, we have $\lambda_a (b)= a^{-1} (a \circ_R b)= \phi_{R(a)}( b)$ for all $a, b \in H$, and hence $\Ker(\lambda)= \Ker(\phi~R)$. Similarly, 
\begin{eqnarray*}
\Fix(\lambda) &=& \{x \in H \mid \lambda_a(x)=x \quad \textrm{for all} \quad a \in H \}\\
&= &\{x \in H \mid  \phi_{R(a)}(x)=x \quad \textrm{for all} \quad a \in H \}\\
&= &\{x \in H \mid  \phi_{y}(x)=x \quad \textrm{for all} \quad y \in \im(R) \}\\
&= &\Fix(\phi|), 
\end{eqnarray*}
which proves our assertion.

\item Let $(H/\Z^{\phi|}_{R|}(H) )_{\overline{R}}$ be the skew left brace induced by  $\I(H, G, \phi, R)/ \Z(\I(H, G, \phi, R))$. It follows from assertion (2) that the additive groups of the skew left braces $(H/\Z^{\phi|}_{R|}(H) )_{\overline{R}}$ and $H_{R|}/ \Ann(H_{R|})$ are the same. Let $\circ_{\overline{R}}$ and $\overline{\circ_R}$ denote the multiplicative group operations in $(H/\Z^{\phi|}_{R|}(H) )_{\overline{R}}$ and $H_{R|}/ \Ann(H_{R|})$, respectively. Then, for $\overline{h}_1, \overline{h}_2 \in H/\Z^{\phi|}_{R|}(H)$, we have 
$$\overline{h}_1 \circ_{\overline{R}} \overline{h}_2=\overline{h}_1 \overline{\phi}_{\overline{R}(\overline{h}_1)} (\overline{h}_2 )= \overline{h_1 \phi_{R(h_1)} (h_2 )}= \overline{h_1 \circ_R h_2}= \overline{h}_1 \overline{\circ_R}~ \overline{h}_2.$$
This established the assertion.
\end{enumerate}
\end{proof}

\begin{no}
For a relative Rota-Baxter group $(H, G, \phi, R)$, let $H^{(2)}$ denote the subgroup of $H$ generated by set  $\{\phi_{g}(h)h^{-1} \mid  h \in H  \mbox{ and }  g \in G \}$. 
\end{no}
 With this setting, we have
\begin{prop}
$(H^{(2)}, G, \phi|, R|)$ is an ideal of $(H, G, \phi, R)$.
\end{prop}

\begin{proof}
Let $g \in G$ and $x = \phi_{g_1}(h)h^{-1} \in H^{(2)}$, where $h \in H$ and $g_1 \in G$. Then
$$ \phi_g(x)= \phi_{g g_1}(h) \phi_{g}(h^{-1})=  \phi_{g g_1 g^{-1}}(\phi_{g}(h)) \phi_{g}(h^{-1})= \phi_{g_2}(h_1) h^{-1}_1,$$
where $g_2=g g_1 g^{-1}$ and $h_1=\phi_{g}(h)$. Thus, $H^{(2)}$ is invariant under the action $\phi$.
\par
We  now establish normality of $H^{(2)}$ in $H$. Let $h, x \in H$ and $g \in G$. Since $\phi_g$ is an automorphism of $H$, there exists $h_1 \in H$ such that $\phi_{g}(h_1)=x$. Thus, we have
\begin{align*}
	x^{-1} \phi_{g}(h)h^{-1} x=& \; \phi_{g}(h_1^{-1}) \phi_{g}(h)h^{-1} \phi_{g}(h_1)\notag\\
	=& \; \big(\phi_{g }(h_1^{-1} h) \big)  \big(h^{-1} h_1\big) \big(h^{-1}_1 \phi_{g}(h_1)\big)\notag\\
	=& \; \big(\phi_{g }(h_1^{-1} h) \; (h^{-1}_1 h)^{-1}\big)\;  \big(h^{-1}_1 \phi_{g}(h_1)\big) \in H^{(2)},\notag\\
\end{align*}
and hence $H^{(2)} \trianglelefteq H$. The conditions \eqref{I1} and \eqref{I2} holds trivially, and hence $(H^{(2)}, G, \phi|, R|)$ is an ideal of $(H, G, \phi, R)$.
\end{proof}

\begin{no}
Given a relative Rota-Baxter group $(H, G, \phi, R)$, let $(H, G, \phi, R)^{(2)}$ denote the relative Rota-Baxter subgroup  $(H^{(2)}, G, \phi|, R|)$. Further, if $(H, \cdot, \circ)$ is a skew left brace, we use $H^{(2)}$ to denote the ideal of $(H, \cdot, \circ)$ generated by $a^{-1} \cdot (a \circ b) \cdot b^{-1}$ for $a, b \in H$.
\end{no}
Continuing with this setting, we have the following result.

\begin{prop}
Let $(H, G, \phi, R)$ be a relative Rota-Baxter group. Then the following hold:
\begin{enumerate}
\item Let $K$ be any normal subgroup of $H$ containing $H^{(2)}$. Then for every ideal of the form $(K, L,  \phi|, R|)$, the relative Rota-Baxter group $(H, G, \phi, R)/ (K, L,  \phi|, R|)$ is trivial.
 \item The skew left brace induced by $\big(\I(H, G, \phi, R)\big)^{(2)}$ is isomorphic to $H^{(2)}_{R}$.
\end{enumerate} 
\end{prop}

\begin{proof}
Recall that, the induced action $\overline{\phi}: G/L: \Aut(H/K)$ is given by  $\overline{\phi}_{\overline{g}}(\overline{h})=\overline{\phi_g(h)}$ for $\overline{g} \in G/L$ and $\overline{h} \in H/K$. We see that
\begin{align*}
\overline{\phi}_{\overline{g}}(\overline{h})=\overline{\phi_g(h)}=\overline{\phi_g(h)h^{-1} h}=\overline{h}, 
\end{align*}
and hence the relative Rota-Baxter group is trivial, proving assertion (1). The second assertion follows immediately from the definitions.
\end{proof}

\begin{defn}
The commutator of a relative Rota-Baxter group $(H, G, \phi, R)$ is defined to be the relative Rota-Baxter group $(H^{\phi}, G, \phi|, R|)$, where $H^{\phi}$ is the subgroup of $H$ generated by its commutator subgroup $[H,H]$ and $H^{(2)}$. We denote the commutator by $(H, G, \phi, R)^{\prime}$.
\end{defn}

Given a skew left brace $(H, \cdot, \circ)$, the commutator $H^\prime$ of $(H, \cdot, \circ)$ is the subgroup of $H^{(\cdot)}$ generated by the commutator subgroup of $H^{(\cdot)}$ and $H^{(2)}$. The commutator $H^\prime$ turns out to be an ideal of $(H, \cdot, \circ)$. The following observations are immediate.

\begin{prop}\label{prop:induced-slb}
Let $(H, G, \phi, R)$ be a relative Rota-Baxter group. Then the following hold:
\begin{enumerate}
\item The commutator $(H, G, \phi, R )^{\prime}$ is an ideal of  $(H, G, \phi, R)$.
\item The skew left brace induced by $(\I (H, G, \phi, R ))^{\prime}$ is isomorphic to $H^{\prime}_{R|}$.
\end{enumerate}
\end{prop}

\begin{lemma}
Let $(H, G, \phi, R)$ be a relative Rota-Baxter group. Then there are maps $\omega_H, \omega^{\phi}_H: \big( H/ \Z^{\phi}_{R}(H)\big) \times \big(H/ \Z^{\phi}_{R}(H)\big)  \rightarrow H^\phi$ defined as
$$\omega_{H}(\overline{h}_1, \overline{h}_2)=[h_1, h_2]$$
and
$$\omega^{\phi}_{H}(\overline{h}_1, \overline{h}_2)= \phi_{R(h_1)}(h_2) h_2^{-1}.$$
\end{lemma}
\begin{proof}
It is easy to see that $\omega_H$ is well-defined. To prove the well-definedness of $\omega^{\phi}_{H}$, let $\overline{h}_1=\overline{h}_3$ and $\overline{h}_2=\overline{h}_4$ in $H/ \Z^{\phi}_{R}(H)$. Then, there exist $z_1, z_2 \in \Z^{\phi}_{R}(H)$ such that $h_1=h_3z_1$ and $h_2=h_4z_2$. This gives
\begin{align}\label{wd1}
	\omega^{\phi}_{H}(\overline{h}_1, \overline{h}_2)=  \phi_{R(h_1)}(h_2) h_2^{-1} =	\phi_{R(h_3 z_1)}(h_4 z_2) (h_4 z_2)^{-1}.
\end{align}
By definition of $\Z^{\phi}_{R}(H)$, we  have $ \phi_{g}(z_1)=z_1$ for all $g \in G$, and hence  $$R(h_3 z_1)=R(h_3 \phi_{R(h_3)}(z_1))=R(h_3) R(z_1).$$
Using the value of $R(h_3 z_1)$ in \eqref{wd1}, we obtain
\begin{align*}
	\omega^{\phi}_{H}(\overline{h}_1, \overline{h}_2)=&\; \phi_{R(h_3) R(z_1)} (h_4 z_2) (h_4 z_2)^{-1}\\
	= &\; \phi_{R(h_3 )}(h_4) z_2 (h_4 z_2)^{-1}  \quad (\text{since}~  z_1 \in \Ker(\phi ~ R)~ \text{and} ~   z_2 \in \Fix(\phi)) \\
	= & \;	\phi_{R(h_3 )}(h_4) h_4^{-1}  \quad (\text{since} ~ z_2 \in \Z(H) )\\
	= & \; \omega^{\phi}_{H}(\overline{h}_3, \overline{h}_4).
\end{align*}
This shows that $\omega^{\phi}_{H}$ is well-defined.
\end{proof}

\begin{defn}
Two relative Rota-Baxter groups  $(H, G, \phi, R)$ and $(K, L, \varphi, S)$ are isoclinic  if there are isomorphisms of relative Rota-Baxter groups 
	$$(\psi_1, \eta_1): (H, G, \phi, R) / \Z(H, G,  \phi, R ) \rightarrow (K, L, \varphi, S) / \Z(K, L, \varphi, S)$$
	and 
	$$(\psi_2, \eta_2): (H, G, \phi, R)^\prime \rightarrow (K, L, \varphi, S)^\prime$$
	such that the following diagram commutes
\begin{align}\label{cd}
\begin{CD}
		H^{\phi}  @<{\omega_H}<<(H/ \Z^{\phi}_{R}(H)) \times (H/ \Z^{\phi}_{R}(H))@>{\omega^{\phi}_H}>> H^{\phi} \\
		@V{\psi_2}VV @V{\psi_1 \times \psi_1}VV @V{\psi_2}VV\\
		K^{\varphi} @<{\omega_K}<< (K/ \Z^{\varphi}_{S}(K)) \times (K/ \Z^{\varphi}_{S}(K))@>{\omega^{\varphi}_K}>> K^{\varphi}.
	\end{CD}
\end{align}
\end{defn}

When the actions $\phi:G \to \Aut(H)$ and $\varphi:L \to \Aut(K)$ are trivial, then the preceding definition boils down to the usual definition of isoclinism of groups $H$ and $K$.

\begin{prop}
An isoclinism of relative Rota-Baxter groups  $(H, G, \phi, R)$ and $(K, L, \varphi, S)$ induces an isoclinism of groups $H$ and $K$.

\end{prop}
\begin{proof}
By definition of isoclinism of relative Rota-Baxter groups, there exist isomorphisms $\psi_1: H/ \Z_R^{\phi}(H) \rightarrow K/ \Z_S^{\varphi}(K)$ and $\psi_2: H^{\phi} \rightarrow K^\varphi$ such that the diagram  \eqref{cd}  commutes. It follows from the commutativity of \eqref{cd} that $\psi_2|_{[H,H]} :[H, H] \rightarrow [K,K]$ and  $\psi_2^{-1}|_{[K,K]} :[K, K] \rightarrow [H,H]$, and hence $\psi_2|_{[H,H]}$ is  an isomorphism. 
\par
Let $h\in \Z(H)$ and $x \in H$. It follows from the commutativity of \eqref{cd} that $\omega_K(\psi_1(\overline{h}), \psi_1(\overline{x}))=\psi_2([h,x])=\psi_2(1)=1$. This implies that $\psi_1(\Z(H)/ \Z^{\phi}_R(H)) \subseteq \Z(K)/ \Z^{\varphi}_S(K)$. Again, commutativity of \eqref{cd} and the fact that $\psi_1$ is an isomorphism implies that $\psi_1(\Z(H)/ \Z^{\phi}_R(H)) = \Z(K)/ \Z^{\varphi}_S(K)$. Now, we have an isomorphism $\psi_1: H/ \Z_R^{\phi}(H) \rightarrow K/ \Z_S^{\varphi}(K)$ such that $\psi_1( \Z(H)/ \Z^{\phi}_R(H) )= \Z(K)/ \Z^{\varphi}_S(K)$. By the third isomorphism theorem, it follow that there is an induced isomorphism $\overline{\phi}_1: H/ \Z(H)$ and $K/ \Z(K)$. Further, the commutativity of the diagram \eqref{cd} implies that the induced diagram
	$$\begin{CD}
		(H/ \Z(H)) \times (H/ \Z(H)) @>\omega_H>> [H,H]\\
@V{\overline{\phi}_1 \times \overline{\phi}_1}VV @V{\psi_2|_{[H, H]}}VV\\
		(K/ \Z(K)) \times (K/ \Z(K)) @>\omega_K>> [K,K]
	\end{CD}$$
also commutes, which is desired.
\end{proof}

\begin{no}
Given a relative Rota-Baxter group $(H,G,\phi,R)$, let $H^{R, \phi}$ denote the subgroup of $H$ generated by the set $\{\phi_{R(h_1)}(h_2) h_2^{-1} \mid h_1, h_2 \in H \}$.
\end{no}
\begin{prop}\label{indiso}
Let $(H,G,\phi,R)$ and $(K,L,\varphi,S)$ be isoclinic relative Rota-Baxter groups. Then the following hold:
\begin{enumerate}
\item $\IM(R)/ (\IM(R) \cap \Ker(\phi)) \cong \IM(S)/( \IM(S) \cap \Ker(\varphi))$.
\item $H^{R, \phi} \cong K^{S, \varphi}$.
\item  $H/ \Z^{\phi|}_{R|}(H) \cong K/ \Z^{\varphi|}_{S|}(K)$.
\item $\I(H,G,\phi,R)$ and $\I(K,L,\varphi,S)$ are isoclinic relative Rota-Baxter groups.
\end{enumerate}

\end{prop}
\begin{proof}
Since $(H,G,\phi,R)$ and $(K,L,\varphi,S)$ are isoclinic, there exist isomorphisms
 $$(\psi_1, \eta_1): (H, G, \phi, R) / \Z(H, G,  \phi, R ) \rightarrow (K, L, \varphi, S) / \Z(K, L, \varphi, S)$$ 
 and
$$(\psi_2, \eta_2): (H, G, \phi, R)^\prime \rightarrow (K, L, \varphi, S)^\prime$$
such that the diagram \eqref{cd} commutes.
\begin{enumerate} 
\item For $h\in H$, let $[R(h)]$ denote the image of $R(h)$ in $\IM(R)/ (\IM(R) \cap \Ker(\phi))$, and $\overline{R(h)}$ denote the image of $R(h)$ in $G/ \Ker(\phi)$. Then, the map $[R(h)] \mapsto \overline{R(h)}$ is an embedding of $\IM(R)/(\IM(R) \cap \Ker(\phi))$ into $G/ \Ker(\phi)$. Similarly, we have an embedding of $\IM(S)/(\IM(S) \cap \Ker(\varphi))$ into $K/ \Ker(\varphi)$. Since $(\psi_1, \eta_1)$ is a homomorphism of relative Rota-Baxter groups, we have  $\eta_1 \; \overline{R} = \overline{S} \; \psi_1$. This identity implies that $\eta_1$ maps $\IM(R)/ (\IM(R) \cap \Ker(\phi))$ onto $\IM(S)/ (\IM(S) \cap \Ker(\varphi))$.

\item Commutativity of the diagram \eqref{cd} gives
\begin{align}\label{com1}
\psi_2(\phi_{R(h_1)}(h_2) h_2^{-1})=\varphi_{S(k_1)}(k_2) k_2^{-1}
\end{align}
for all $h_1, h_2 \in H$, where $\psi_1(\overline{h}_1)=\overline{k}_1$ and $\psi_1(\overline{h}_2)=\overline{k}_2$ for some $k_1, k_2 \in K$. This shows that $\psi_2$ maps $H^{R, \phi}$ into $K^{S, \varphi}$. We claim that $\psi_2|_{H^{R, \phi}}: H^{R, \phi} \to K^{S, \varphi}$ is surjective. Let $\varphi_{S(k_1)}(k_2) k_2^{-1} \in K^{S, \varphi}$, where $k_1, k_2 \not\in \Z^{\varphi}_{S}(K)$. By leveraging the fact that $\psi_1$ is an isomorphism, we can identify $h_1, h_2 \in H$ such that $\psi_1(h_1)=k_1$ and $\psi_2(h_2)=k_2$. Substituting these values into \eqref{com1}, we obtain $\psi_2(\phi_{R(h_1)}(h_2) h_2^{-1})=\varphi_{S(k_1)}(k_2) k_2^{-1}$, which establishes our claim.

\item Let $ \Z^{\phi|}_{R|}(H)$ and $ \Z^{\varphi|}_{S|}(K)$ denote the first tuple of center of  $\I(H,G,\phi,R)$ and $\I(K,L,\varphi,S)$, respectively. Observe that $\Z^{\phi}_{R}(H) \subseteq \Z^{\phi|}_{R|}(H)$ and $\Z^{\varphi}_{S}(K) \subseteq \Z^{\varphi|}_{S|}(K)$. Further, we have an isomorphism $\psi_1:H/ \Z^{\phi}_R(H) \to K/ \Z^{\varphi}_S(K)$ of groups. We claim that $\Z^{\phi|}_{R|}(H)/ \Z^{\phi}_{R}(H)$ and $\Z^{\varphi|}_{S|}(K)/ \Z^{\varphi}_{S}(K)$ are isomorphic under the restriction of $\psi_1$. To see this, let $x \in \Z^{\phi|}_{R|}(H)$ and let $\psi_1(\overline{x})= \overline{k}$ for some $k \in K$. Since $(\psi_1, \eta_1)$ is a homomorphism of relative Rota-Baxter groups, for all $g \in G$, we have 
\begin{equation}\label{morphism eq}
\overline{S}\psi_1=\eta_1 \overline{R} \quad \textrm{and} \quad \psi_1 \overline{\phi}_{\overline{g}}= \overline{\varphi}_{\eta_1(\overline{g})} \psi_1.
\end{equation}
Since $x \in \Fix(\phi|)$, we have $\phi_{R(h)}(x)=x$ for all $h \in H$. Using \eqref{morphism eq}, we obtain
$$\overline{k}=\psi_1(\overline{x})=\psi_1\overline{\phi}_{\overline{R}(\overline{h})}(\overline{x})= \overline{\varphi}_{\eta_1(\overline{R}(\overline{h}))} \psi_1(\overline{x})= \overline{\varphi}_{\overline{S}(\psi_1(\overline{h}))} (\overline{k}),$$
and hence $k \in \Fix (\varphi|)$ modulo  $\Z^{\varphi}_{S}(K)$. Given any $k_1 \in K$, let $\psi_1(\overline{x}_1)=\overline{k}_1$. Since $x \in \Z(H)$, we have $$\overline{k} ~\overline{k}_1=\psi_1(\overline{x}) \psi_1(\overline{x}_1)=\psi_1(\overline{xx_1})=\psi_1(\overline{x_1x})= \psi_1(\overline{x}_1) \psi_1(\overline{x})=\overline{k}_1\overline{k},$$
and hence $k \in \Z(K)$ modulo  $\Z^{\varphi}_{S}(K)$. Finally, since $x \in \Ker(\phi ~R)$, using \eqref{morphism eq}, we obtain $k \in \Ker(\varphi ~S)$ modulo $\Z^{\varphi}_{S}(K)$. Thus, we can deduce that $\psi_1(\overline{x})$ belongs to $\Z^{\varphi|}_{S|}(K)/ \Z^{\varphi}_{S}(K).$
 For surjectivity of the restriction of $\psi_1$,  let $y \in \Z^{\varphi|}_{S|}(K)$ and $h \in H$ be such that $\psi_1(\overline{h}) = \overline{y}$. Using a similar argument as before, we can show that $h \in \Z^{\phi|}_{R|}(H)$. This proves our claim. Now, it follows from the third isomorphism theorem that $\psi_1$ induces an isomorphism $H/ \Z^{\phi|}_{R|}(H) \cong K/ \Z^{\varphi|}_{S|}(K)$, which is desired.

\item By definition, we have 
$$\I(H,G,\phi,R)/ \Z(\I(H,G,\phi,R))= (H/\Z^{\phi|}_{R|}(H), \IM(R)/(\IM(R) \cap \Ker(\phi)), \overline{\phi}, \overline{R} ),$$ 
$$(\I(H,G,\phi,R))^\prime=(H^{\phi|}, \IM(R), \phi|, R|),$$
$$\I(K,L,\varphi,S)/ \Z(\I(K,L,\varphi,S))= (K/\Z^{\varphi|}_{S|}(K), \IM(S)/(\IM(S) \cap \Ker(\varphi)), \overline{\varphi}, \overline{S})$$
 and 
 $$(\I(K,L,\varphi,S))^\prime=(K^{\varphi|}, \IM(S), \varphi|, S|).$$ It follows from assertions (1) and (3) that
$$\I(H,G,\phi,R)/ \Z(\I(H,G,\phi,R)) \cong \I(K,L,\varphi,S)/ \Z(\I(K,L,\varphi,S))$$
via an isomorphism induced by the pair $(\psi_1,\eta_1)$. Similarly, it follows from  assertion (2) and the fact $(H,G,\phi,R)^\prime \cong (K,L,\varphi,S))^\prime$ that $(\I(H,G,\phi,R))^\prime \cong (\I(K,L,\varphi,S))^\prime$ via an isomorphism induced by the pair $(\psi_2,\eta_2)$. Now, it follows from the commutativity of the diagram \eqref{cd} that the relative Rota-Baxter groups $\I(H,G,\phi,R)$ and  $\I(K,L,\psi,S)$ are isoclinic.
\end{enumerate}
\end{proof}

Isoclinism of skew left braces has been introduced recently in \cite{TV23}. Let $(H, \cdot, \circ)$ be a skew left brace. Then we have maps  $\theta_H, \theta^{*}_H: (H/ \Ann(H)) \times (H/ \Ann(H)) \to H^\prime$ given by $$\theta_H(\overline{a}, \overline{b})= a \cdot b \cdot a^{-1} \cdot b^{-1}$$ and $$\theta^{*}_H(\overline{a}, \overline{b})= a^{-1} \cdot (a \circ b) \cdot a^{-1}.$$ Two skew left braces $(H, \cdot, \circ)$ and $(K, \cdot, \circ)$ are said to be isoclinic if there are isomorphisms $\xi_1 : H/ \Ann(H)\to K/ \Ann(K)$ and $\xi_2: H^\prime \to K^\prime$ such that the following diagram commutes
\begin{align}\label{cdsb}
\begin{CD}
		H^\prime  @<{\theta_H}<<(H/ \Ann(H)) \times (H/ \Ann(H))@>{\theta^{*}_H}>> H^\prime \\
		@V{\xi_2}VV @V{\xi_1 \times \xi_1}VV @V{\xi_2}VV\\
		K^\prime @<{\theta_K}<< (K/ \Ann(K)) \times (K/ \Ann(K))@>{\theta^{*}_K}>> K^\prime.
	\end{CD}	
\end{align}

We conclude with the following result.

\begin{thm}\label{isoclinism rrbg implies isoclinism slb}
Let $(H,G,\phi,R)$ and $(K,L,\varphi,S)$ be isoclinic relative Rota-Baxter groups. Then their induced skew left braces are also isoclinic.
\end{thm}

\begin{proof}
By Proposition \ref{indiso}, we have that $\I(H,G,\phi,R)$ and $\I(K,L,\varphi,S)$ are also isoclinic. Thus, there exist isomorphisms
$$(\psi_1, \eta_1): \I(H, G, \phi, R) / \Z(\I(H, G,  \phi, R) ) \rightarrow \I(K, L, \varphi, S) / \Z(\I(K, L, \varphi, S))$$ 
and
$$(\psi_2, \eta_2): (\I(H, G, \phi, R))^\prime \rightarrow (\I(K, L, \varphi, S))^\prime$$
such that the following diagram commutes
\begin{align}
\begin{CD}\label{cd3}
	H^{\phi|}  @<{\omega_H}<<\big( H/ \Z^{\phi|}_{R|}(H)\big) \times \big( H/ \Z^{\phi|}_{R|} (H)\big)@>{\omega^{\phi|}_H}>> H^{\phi|} \\
	@V{\psi_2}VV @V{\psi_1 \times \psi_1}VV @V{\psi_2}VV\\
	K^{\varphi|} @<{\omega_K}<< \big(K/ \Z^{\varphi|}_{S|} (K)\big) \times \big(K/ \Z^{\varphi|}_{S|}(K) \big)@>{\omega^{\varphi|}_K}>> K^{\varphi|}.
\end{CD}
\end{align}
Since every morphism of relative Rota-Baxter groups induces a morphism of corresponding skew left braces, diagram \ref{cd3} gives the following commutative diagram of induced skew left braces

\begin{align}
\begin{CD}\label{cd4}
	(H^{\phi|})_{R|}  @<{\omega_H}<< (H/ \Z^{\phi|}_{R|}(H))_{\overline{R}}\times  (H/ \Z^{\phi|}_{R|}(H))_{\overline{R}}@>{\omega^{\phi|}_H}>> (H^{\phi|})_{R|} \\
	@V{\psi_2}VV @V{\psi_1 \times \psi_1}VV @V{\psi_2}VV\\
	(K^{\varphi|})_{S|} @<{\omega_K}<< (K/ \Z^{\varphi|}_{S|}(K))_{\overline{S}} \times  (K/ \Z^{\varphi|}_{S|}(K))_{\overline{S}}  @>{\omega^{\varphi|}_K}>> (K^{\varphi|})_{S|}.
\end{CD}
\end{align}

Proposition \ref{lem:mod-iso} shows that the skew left brace induced by  $\I(H, G, \phi, R) / \Z(\I(H, G,  \phi, R) )$  on $H/ \Z^{\phi|}_{R|}(H)$ is identical to $H_{R|} / \Ann(H_{R|})$.  Further, Proposition \ref{prop:induced-slb} asserts that the skew left brace induced by $(\I(H,G,\phi,R))^{\prime}$ on $H^{\phi|}$ is identical to the skew left brace $H_{R|}^{\prime}$. Since $H_R=H_{R|}$ and  $K_S=K_{S|}$, diagram \ref{cd4} gives the following commutative diagram of skew left braces
$$
\begin{CD}
	H_R^\prime  @<{\omega_H}<< (H_R/\Ann(H_R))\times (H_R/\Ann(H_R))@>{\omega^{\phi|}_H}>> H_R^\prime\\
	@V{\psi_2}VV @V{\psi_1 \times \psi_1}VV @V{\psi_2}VV\\
	K_S^\prime @<{\omega_K}<< (K_S/\Ann(K_S)) \times  (K_S/\Ann(K_S))  @>{\omega^{\varphi|}_K}>> K_S^\prime.
\end{CD}
$$
which shows that $H_R$ and $K_S$ are isoclinic.
\end{proof}

\begin{remark}
The converse of the preceding theorem is not necessarily true. For instance, take two non-isomorphic trivial braces (abelian groups) $(H, \cdot_H, \cdot_H)$ and $(K, \cdot_K, \cdot_K)$. Then they are  isoclinic as braces since $H/ \Ann(H)$, $K/ \Ann(K)$, $H^\prime $ and $K^\prime$ are all trivial. But, the relative Rota-Baxter groups $(H^{(\cdot_H)}, H^{(\cdot_H)}, \lambda^H, \Id_H)^\prime$ and $(K^{(\cdot_K)}, K^{(\cdot_K)}, \lambda^K, \Id_K)^\prime$ cannot be isomorphic since $H^{(\cdot_H)}$ and $K^{(\cdot_K)}$ are not isomorphic.
\end{remark}

\begin{ack}
The authors thank Tushar Kanta Naik for his interest in this work. Nishant Rathee thanks IISER Mohali for the institute post doctoral fellowship. Mahender Singh is supported by the SwarnaJayanti Fellowship grants DST/SJF/MSA-02/2018-19 and SB/SJF/2019-20.
\end{ack}

\section{Declaration}
The authors declare that they have no conflict of interest.

\section{Data Availability}
The authors can make the GAP code used in this study available upon request. 


\medskip

 
 
\begin{thebibliography}{99}
	
\bibitem{DB18}
D.~ Bachiller, \textit{Extensions, matched products, and simple braces}, J. Pure Appl. Algebra 222 (2018),  1670--1691.

\bibitem{BCJO18}
D.~Bachiller, F.~Cedo, E.~ Jespers and J.~ Okninski,  \textit{Iterated matched products of finite braces and simplicity; new solutions of the Yang-Baxter equation}, 
Trans. Amer. Math. Soc. 370 (2018), 4881--4907.

\bibitem{BCJO19} D.~Bachiller, F.~Cedo, E.~ Jespers and J.~ Okninski, \textit{Asymmetric product of left braces and simplicity; new solutions of the Yang-Baxter equation}, Commun. Contemp. Math. 21 (2019), 1850042, 30 pp.
	
\bibitem{VV2022} V.~Bardakov and V.~Gubarev, \textit{Rota-Baxter groups, skew left braces, and the Yang-Baxter equation},  J. Algebra 587 (2022), 328--351.

\bibitem{VV2023} V.~Bardakov and V.~Gubarev, \textit{Rota--Baxter operators on groups}, Proc. Indian Acad. Sci. Math. Sci. 133 (2023), no. 1, Paper No. 4, 29 pp.

\bibitem{VVY022}
V. ~Bardakov, M.~ Neshchadim, M. K.~ Yadav, \textit{Computing skew left braces of small orders}, Internat. J. Algebra Comput. 30 (2020), no. 4, 839–851.

\bibitem{BGST}
C.~Bai, L.~ Guo, Y.~Sheng  and R.~Tang, \textit{Post-groups, (Lie-)Butcher groups and the Yang--Baxter equation}, Math. Ann. (2023), to appear.

\bibitem{AC20}
A.~Caranti, \textit{Bi-skew braces and regular subgroups of the holomorph}, J. Algebra, \textbf{562}, (2020), 647-665


\bibitem{CCS}
F.~Catino, I.~ Colazzo and P.~ Stefanelli,  \textit{Regular subgroups of the affine group and asymmetric product of braces}, J. Algebra 455 (2016), 164--182.

\bibitem{CCS1}
F.~Catino, I.~ Colazzo and P.~ Stefanelli,  \textit{The matched product of set-theoretical solutions of the Yang-Baxter equation}, J. Pure Appl. Algebra 224 (2020), 1173--1194.

\bibitem{CCS2}
F.~Catino, I.~ Colazzo and P.~ Stefanelli, \textit{The matched product of the solutions to the Yang-Baxter equation of finite order},  Mediterr. J. Math. 17 (2020), Paper No. 58, 22 pp.

\bibitem{CCS3}
F.~Catino, I.~ Colazzo and P.~ Stefanelli,  \textit{Skew left braces with non-trivial annihilator}, J. Algebra Appl. 18 (2019),  No.2, 1950033.


\bibitem{CMP}
F.~ Catino, M.~Mazzotta and P.~Stefanelli, \textit{Rota-Baxter operators on Clifford semigroups and the Yang-Baxter equation}, (2022), arXiv:2204.05004v1.

\bibitem{D92} 
V. G.~ Drinfel'd, \textit{On some unsolved problems in quantum group theory}, Quantum Groups, Leningrad, 1990, Lecture Notes in Math., vol. 1510, Springer, Berlin, 1992, pp. 1-8.

\bibitem{GV17}
L.~Guarnieri and L.~ Vendramin, \textit{Skew braces and the Yang-Baxter equation},  Math. Comp. 86 (2017),  2519--2534.

\bibitem{LHY2021} 
L.~Guo, H.~Lang and Y.~ Sheng, \textit{Integration and geometrization of Rota-Baxter Lie algebras}, Adv. Math. 387 (2021), 107834, 34 pp.


\bibitem{MR0003389} P. Hall, \textit{The classification of prime-power groups}, J. Reine Angew. Math. 182 (1940), 130--141. 

\bibitem{JYC}
J.~ Jiang, Y.~ Sheng and C.~ Zhu, \textit{Lie theory and cohomology of relative Rota-Baxter operators}, (2021), arXiv:2108.02627v3.

\bibitem{LV16}
V.~Lebed and L.~Vendramin, \textit{Cohomology and extensions of braces}, Pacific J. Math.  284 (2016),  191--212.

\bibitem{TV23}
 T.~ Letourmy and L.~Vendramin, \textit{Isoclinism of skew braces}, (2022), arXiv:2211.14414.

\bibitem{NMY1}
N.~ Rathee and M. K. ~Yadav, \textit{Cohomology, extensions and automorphisms of skew braces}, (2021), arXiv:2102.12235.


\bibitem{WR07}
W.  ~Rump,  \textit{Braces, radical rings,  and the quantum Yang Baxter equation}, J. Algebra 307 (2007),  153--170.

\bibitem{WR08}
W.~Rump, \textit{Semidirect products in algebraic logic and solutions of the quantum Yang-Baxter equation}, J. Algebra Appl. 7 (2008),
471--490.

\bibitem{GAP} The GAP Group,  \textit{Groups Algorithms and Programming},  version 4.11.0 (2020), http://www.gap-system.org.
	
\end{thebibliography}
\end{document}